\documentclass[12pt]{article}
\usepackage{amsfonts}
\usepackage{graphicx}
\usepackage{amsmath}
\usepackage[onehalfspacing]{setspace}

\numberwithin{equation}{section}

\newtheorem{theorem}{Theorem}[section]

\newtheorem{lemma}[theorem]{Lemma}

\newenvironment{proof}[1][Proof]{\textbf{#1.} }{\ \rule{0.5em}{0.5em}}

\begin{document}

\title{A Note on Indefinite Stochastic Riccati Equations}
\date{}
\author{\emph{{\textsc{\ Zhongmin Qian\thanks{%
Mathematical Institute, University of Oxford, 24-29 St Giles', Oxford OX1
3LB, UK.} and Xun Yu Zhou\thanks{%
Mathematical Institute, University of Oxford, 24-29 St Giles', Oxford OX1
3LB, UK, and Department of Systems Engineering \& Engineering Management,
The Chinese University of Hong Kong, Shatin, Hong Kong. X.~Y. Zhou
acknowledges the support from a start-up fund of Oxford, as well as from the
Nomura Centre and OMI.} }}}}
\maketitle


\leftskip1truecm \rightskip1truecm \noindent {\textbf{Abstract.}} An
indefinite stochastic Riccati Equation is a matrix-valued, highly nonlinear
backward stochastic differential equation together with an algebraic, matrix
positive definiteness constraint. We introduce a new approach to solve a
class of such equations (including the existence of solutions) driven by
one-dimensional Brownian motion. The idea is to replace the original
equation by a system of BSDEs (without involving any algebraic constraint)
whose existence of solutions automatically enforces the original algebraic
constraint to be satisfied.

\leftskip0truecm \rightskip0truecm

\vskip0.5truecm

\noindent \textit{Key words.} Stochastic Riccati equation, indefinite
matrix, backward stochastic differential equation, stochastic differential
equation

\vskip0.5truecm

\noindent \textit{AMS Classification.} 60H10, 60H30, 60J45

\newpage

\section{Introduction}

Stochastic matrix Riccati equations were first introduced by Bismut \cite%
{MR0406663} in his study of some stochastic control problems. A very special
class of these equations is the so-called quadratic backward stochastic
differential equation (BSDE). The existence and uniqueness of solutions for
such BSDEs remain a largely open problem, particularly for BSDE systems; but
see \cite{MR1782267}, \cite{MR1944142}, \cite{MR2391164} and the references
therein for recent progress. The \textit{indefnite} stochastic Riccati
equations (SRE) were first formulated in \cite{MR1626817}, motivated by the
introduction of the indefinite stochastic linear--quadratic control
problems. Such an equation is typically matrix-valued, highly nonlinear (not
even quadratic), and involves a positive semidefinite constraint in addition
to the backward equation. In \cite{MR1982738}, the uniqueness of solutions
to the SRE was established in the greatest generality based on a control
argument, but the existence was solved only for several very special cases.
The general existence remains to this date a significant open problem.

The SRE is a BSDE over a running time interval $[0,T]$: 
\begin{eqnarray}
dP &=&\sum_{j=1}^{k}\Lambda _{j}dW^{j}-\left[ PA+A^{\prime
}P+\sum_{j=1}^{k}\left( \Lambda _{j}C_{j}+C_{j}^{\prime }\Lambda
_{j}+C_{j}^{\prime }PC_{j}\right) +Q\right] dt  \notag \\
&&+\left[ PB+\sum_{j=1}^{k}\left( C_{j}^{\prime }P+\Lambda _{j}\right) D_{j}%
\right] K^{-1}\left[ B^{\prime }P+\sum_{j=1}^{k}D_{j}^{\prime }\left(
PC_{j}+\Lambda _{j}\right) \right] dt  \label{p1-eq1}
\end{eqnarray}%
subject to the constraint that 
\begin{equation}
K=R+\sum_{j=1}^{k}D_{j}^{\prime }PD_{j}>0  \label{p2-eq2}
\end{equation}%
in the matrix sense, and subject to the terminal condition that $P(T)=H$
which is $\mathcal{F}_{T}$-measurable. In this formulation, the time
parameter $t$ is omitted for simplicity, the capital letters $%
A,B,C,D,\Lambda $, $Q$ and $P $ are real matrix valued (adapted) processes
and $D^{\prime }$ means the transpose of $D$ etc. All the matrix processes
are square with the same dimension $n$, $P$ and $\Lambda _{j}$, $%
j=1,2,\cdots,k$, are unknowns and all the other parameters are given, and $%
W=(W^{1},\cdots ,W^{k})$ is $k$-dimensional standard Brownian motion.

The given matrix $R$ in defining $K$ is called the gauge matrix, which is an
adapted process. The SRE (\ref{p1-eq1})--(\ref{p2-eq2}) is indefinite, if
the gauge matrix $R $ is allowed to be indefinite, i.e., $R$ can have zero
or negative eigenvalues.

The problem is to look for square integrable adapted processes $P$ and $%
(\Lambda _{j})$ satisfying the corresponding stochastic integral equations
as well as the constraint (\ref{p2-eq2}). Moreover, in view of the proved
uniqueness of solutions, a solution matrix $P$ must be symmetric as long as
the parameters $Q$, $R$ and $H$ are symmetric.

We believe that the existence for the general SRE (\ref{p1-eq1})--(\ref%
{p2-eq2}) in high dimensions will remain to be an open question for some
time. The main challenge, apart from the highly nonlinear nature of the BSDE
(\ref{p1-eq1}) and the fact that the equation is matrix-valued, stems from
the presence of an additional algebraic constraint (\ref{p2-eq2}). In this
note, we develop a new approach to solve a class of SREs driven by
one-dimensional Brownian motion. The main idea is to consider a system of
BSDEs satisfied by $(K,K^{-1})$ without any algebraic constraint. It turns
out that the existence of solutions to the equation satisfied by $K^{-1}$
can be established independently, which in turn will ensure the validity of
the original constraint $K>0$.

\section{The main result}

In this paper we consider the SRE driven by a one-dimensional Brownian
motion, with the matrix process $D$ being invertible.

Therefore, in the remainder of this note, $W$ is a \emph{one-dimensional}
standard Brownian motion on a complete probability space $(\Omega ,\mathcal{F%
},\mathbb{P})$, and $(\mathcal{F}_{t})_{t\geq 0}$ is the Brownian filtration
generated by $W$. Without lose of generality, we may assume that $D=I$, and
we study the following BSDE 
\begin{equation}  \label{bs-01}
\begin{array}{rl}
dP = & \Lambda dW-\left[ PA+A^{\prime }P+C^{\prime }PC+\Lambda C+C^{\prime
}\Lambda +Q\right] dt, \\ 
& \;\;+\left[ PB+C^{\prime }P+\Lambda \right] K^{-1}\left[ B^{\prime
}P+PC+\Lambda \right] dt, \;\;t\in[0,T], \\ 
P(T)= & H,%
\end{array}%
\end{equation}%
where $T>0$ is fixed throughout the paper and the random symmetric matrix $H$
is bounded and $\mathcal{F}_{T}$-measurable, subject to the constraint that 
\begin{equation}
K\equiv R+P>0.  \label{co-01}
\end{equation}

For simplicity we assume that the coefficients $A,B,C,Q$ are bounded, square 
$n\times n$ matrix valued adapted processes, and in addition $Q$ is
symmetric.

We are interested in the indefinite gauge case, that is, $R$ may have zero
or negative eigenvalues, not necessarily being positive definite. The most
interesting case in applications is when $R$ is a matrix-valued It\^{o}
process; so we assume that it has the representation 
\begin{equation}
R(t)=R(0)+\int_{0}^{t}Fds+\int_{0}^{t}GdW  \label{r-a1}
\end{equation}%
where $F,G$ are bounded, adapted and measurable symmetric matrix-valued
processes.

By a solution $(P,\Lambda)$ we mean a pair of $(\mathcal{F}_{t})$-adapted,
measurable, square integrable and matrix-valued processes $P=(P(t))_{t\in
\lbrack 0,T]}$ and $\Lambda =(\Lambda (t))_{t\in \lbrack 0,T]}$ such that $%
K(t)=R(t)+P(t)>0$ in the matrix sense for all $t\in \lbrack 0,T]$, a.s., and%
\begin{eqnarray}
P(t) &=&H-\int_{t}^{T}\Lambda dW  \notag \\
&&+\int_{t}^{T}\left[ PA+A^{\prime }P+C^{\prime }PC+\Lambda C+C^{\prime
}\Lambda +Q\right] ds  \notag \\
&&-\int_{t}^{T}\left[ PB+C^{\prime }P+\Lambda \right] K^{-1}\left[ B^{\prime
}P+PC+\Lambda \right] ds  \label{i-bsde-1}
\end{eqnarray}%
for $t\in \lbrack 0,T]$, a.s., where the stochastic integral is understood
in the It\^{o} sense. A solution $(P,\Lambda )$ is called bounded if $P$ is
bounded. Clearly, if $(P,\Lambda )$ is a solution, then $P$ must be a
continuous matrix-valued semimartingale.

Due to the presence of $K^{-1}$ in (\ref{bs-01}), it is natural to rewrite (%
\ref{bs-01}) in terms of $K=R+P$ and $\tilde{\Lambda}=\Lambda +G$. This can
be achieved by making substitutions in (\ref{bs-01}): $P$ by $K-R$ and $%
\Lambda $ by $\tilde{\Lambda}-G$, leading to the following BSDE 
\begin{equation}  \label{bs-02}
\begin{array}{rl}
dK= & \tilde{\Lambda}dW-\left[ K\tilde{A}+\tilde{A}^{\prime }K+\tilde{Q}%
\right] dt+\left[ KB+\tilde{\Lambda}-\tilde{R}\right] K^{-1}\left[ B^{\prime
}K+\tilde{\Lambda}-\tilde{R}^{\prime }\right] dt, \\ 
K(T)= & R(T)+H,%
\end{array}%
\end{equation}%
where%
\begin{equation}
\tilde{Q}=Q+F+C^{\prime }RC+R\left( BC-A\right) +\left( C^{\prime }B^{\prime
}-A^{\prime }\right) R  \label{q-02}
\end{equation}%
and%
\begin{equation}
\tilde{A}=A-BC\text{, \ }\ \tilde{R}=RB+C^{\prime }R+G\text{.}  \label{a-01}
\end{equation}

We are now in a position to state our main result.

\begin{theorem}
\label{main-th}\bigskip Assume

\begin{itemize}
\item[\textrm{(i)}] 
\begin{equation}
\tilde{R}=RB+C^{\prime }R+G=0  \label{c-a}
\end{equation}%
and%
\begin{equation}
\tilde{Q}=Q+F+C^{\prime }RC+R\left( BC-A\right) +\left( C^{\prime }B^{\prime
}-A^{\prime }\right) R\geq 0\text{,}  \label{c-b}
\end{equation}

\item[\textrm{(ii)}] $G$ and $F$ are bounded adapted measurable processes
such that $R(T)+H>0$, and there is a constant $\delta >0$ such that $\left(
R(T)+H\right) ^{-1}\geq \delta I$.
\end{itemize}

Then there is a unique solution $(P,\Lambda )$ to the SRE (\ref{bs-01})--(%
\ref{co-01}). Moreover, $P+R$ is bounded.
\end{theorem}

The uniqueness has been established in \cite{MR1982738}, Theorem 3.2. The
existence, on the other hand, is known in the so-called definite case,
namely, when $R>0,\;Q\geq 0,\; H\geq 0$; see \cite{MR1149069} (where there
is an additional assumption that $D=0$). The existence when all these
matrices are allowed to be indefinite is investigated in \cite{MR1982738}
for several special cases. It should be noted that the existence of (\ref%
{bs-01}) is by no means unconditional; the problem is to find \textit{%
sufficient conditions} under which the existence holds. One of the
conditions, (\ref{c-a}), of Theorem \ref{main-th} stipulates that $R$
satisfies the following It\^{o} equation%
\begin{equation}  \label{keycondition}
dR=-(RB+C^{\prime }R)dW+Fdt
\end{equation}%
where $F$ can be arbitrary (up to the required Lebesgue integrability). 
The other condition, (\ref{c-b}), requires an ``overall" positive
semidefiniteness in place of that of individual matrices.

We will make further comments on the conditions of the preceeding theorem in
Section \ref{example}.

\section{A linear equation}

We need a result about the representation for a linear matrix-valued BSDE.
Consider 
\begin{eqnarray}
dY &=&\sum_{j=1}^{k}U_{j}dW^{j}-\left[ Y\hat{A}+\hat{A}^{\prime
}Y+\sum_{i=1}^{m}\hat{E}_{i}^{\prime }Y\hat{E}_{i}+\sum_{j=1}^{k}\left( U_{j}%
\hat{C}_{j}+\hat{C}_{j}^{\prime }U_{j}+\hat{C}_{j}^{\prime }Y\hat{C}%
_{j}\right) +\hat{Q}\right] dt\text{,}  \notag \\
Y_{T} &=&\hat{H}\text{ ,}  \label{lin-p1}
\end{eqnarray}%
where $\hat{A}$, $\hat{C}$, $\hat{E}$ and $\hat{Q}$ are $n\times n$ matrix
valued, adapted and bounded, $\hat{H}$ is bounded and $\mathcal{F}_{T}$%
-measurable, and $W$ is a standard Brownian motion of dimension $k$. The
BSDE is linear so there is a unique solution. Choose a standard Brownian
motion $\hat{W}$ of dimension $m$, which is independent of $W$. For any $p
\in \mathbb{R}^{n}$ and $0\leq t<T$, let $\xi $ be the solution to the
linear stochastic differential equation%
\begin{equation}
d\xi =\hat{A}\xi ds+\sum_{j=1}^{k}\hat{C}_{j}\xi dW^{j}+\sum_{i=1}^{m}\hat{E}%
_{i}\xi d\hat{W}^{i}\text{, }\xi _{t}=p \text{.}  \label{y-li1}
\end{equation}

\begin{lemma}
\label{lem-p1}Under the above assumptions and notations, we have%
\begin{equation}
p ^{\prime }Y(t)p =\mathbb{E}\left[ \left. \xi _{T}^{\prime }\hat{H}\xi
_{T}+\int_{t}^{T}\xi _{s}^{\prime }\hat{Q}\xi _{s}ds\right\vert \mathcal{F}%
_{t}\right] \text{.}  \label{pos-e1}
\end{equation}
\end{lemma}

\begin{proof}
Applying It\^{o}'s formula to 
\begin{equation}
J(s)=\xi _{s}^{\prime }Y(s)\xi _{s}=(Y(s)\xi _{s},\xi _{s})  \label{j-li1}
\end{equation}%
we obtain%
\begin{equation}
dJ=-\xi ^{\prime }\hat{Q}\xi ds+\xi ^{\prime }\sum_{j=1}^{k}\left( U_{j}+%
\hat{C}_{j}^{\prime }Y+Y\hat{C}_{j}\right) dW^{j}\xi +\xi ^{\prime
}\sum_{i=1}^{m}\left( Y\hat{E}_{i}+\hat{E}_{i}^{\prime }Y\right) d\hat{W}%
^{i}\xi.  \label{j-eq1}
\end{equation}%
Integrating from $t$ to $T$, and conditional on $\mathcal{F}_{t}$ we obtain (%
\ref{pos-e1}).
\end{proof}

\begin{lemma}
\label{lem-p1a}If in addition $\hat{H}\geq \delta I$ for some constant $%
\delta >0$ and $\hat{Q}\geq 0$, then the solution $Y$ to (\ref{lin-p1})
satisfies 
\begin{equation}
Y(t)\geq \delta e^{-\beta (T-t)}I\text{ \ \ \ }\forall t\in \lbrack 0,T],\;\;%
\mbox{a.s.},  \label{bound-1}
\end{equation}%
where 
\begin{equation}
\beta =\text{ess}\sup_{\omega \in \Omega ,s\leq T}\{-2\inf_{|\xi |=1}\xi
^{\prime }\hat{A}(s,\omega )\xi ,0\}\text{.}  \label{er-01}
\end{equation}
\end{lemma}

\begin{proof}
Let $p\in \mathbb{R}^{n}$ and $\xi $ solve (\ref{y-li1}). Applying It\^{o}'s
formula to $|\xi |^{2}$ to obtain 
\begin{eqnarray*}
d|\xi |^{2} &=&2\xi ^{\prime }\hat{A}\xi ds+\sum_{j=1}^{k}|\hat{C}_{j}\xi
|^{2}ds+\sum_{i=1}^{m}|\hat{E}_{i}\xi |^{2}ds \\
&&+2\sum_{j=1}^{k}\xi ^{\prime }\hat{C}_{j}\xi dW^{j}+2\sum_{i=1}^{m}\xi
^{\prime }\hat{E}_{i}\xi d\hat{W}^{i}\text{.}
\end{eqnarray*}%
Integrating over $[t,T]$ and taking conditional expectation on $\mathcal{F}%
_{t}$ we obtain%
\begin{eqnarray*}
\mathbb{E}\left[ \left. |\xi _{s}|^{2}\right\vert \mathcal{F}_{t}\right]
&=&|p|^{2}+\int_{t}^{s}\mathbb{E}\left[ \left. \left( 2\xi ^{\prime }\hat{A}%
\xi +\sum_{j=1}^{k}|\hat{C}_{j}\xi |^{2}+\sum_{i=1}^{m}|\hat{E}_{i}\xi
|^{2}\right) dr\right\vert \mathcal{F}_{t}\right] \\
&\geq &|p|^{2}-\beta \int_{t}^{s}\mathbb{E}\left[ \left. |\xi
_{r}|^{2}\right\vert \mathcal{F}_{t}\right] dr\text{.}
\end{eqnarray*}%
The Gronwall inequality then yields 
\begin{equation*}
\mathbb{E}\left[ \left. |\xi _{T}|^{2}\right\vert \mathcal{F}_{t}\right]
\geq |p|^{2}e^{-\beta (T-t)}.
\end{equation*}%
Finally, it follows from Lemma \ref{lem-p1} that 
\begin{equation*}
p^{\prime }Y(t)p\geq \mathbb{E}\left[ \left. \xi _{T}^{\prime }\hat{H}\xi
_{T}\right\vert \mathcal{F}_{t}\right] \geq \delta |p|^{2}e^{-\beta (T-t)}
\end{equation*}%
which implies (\ref{bound-1}).
\end{proof}

\section{Proof of Theorem \protect\ref{main-th}}

The remainder of the paper is devoted to the proof of Theorem \ref{main-th}.

To handle the positive definiteness constraint (\ref{co-01}), we couple the
BSDE (\ref{bs-02}) together with another BSDE for $X=K^{-1}$, and consider
the resulting system of BSDEs \textit{without} the explicit constraint $K>0$%
. This last constraint will be \textit{implied} by the existence of
solutions to this system of BSDEs. %
%

Therefore we next derive the BSDE for $X=K^{-1}$ which can be obtained from
the identities $XK=KX=I$. In fact, by integrating by parts, 
\begin{equation}
dX=-X(dK)X-\langle dX,dK\rangle X\text{.}  \label{x-e1}
\end{equation}%
In particular the martingale part of $X$ is $-X\tilde{\Lambda}XdW$; so $%
\langle dX,dK\rangle =-X\tilde{\Lambda}X\tilde{\Lambda}dt$. Substituting
this equation into (\ref{x-e1}) to obtain%
\begin{equation}
dX=-X(dK)X+X\tilde{\Lambda}X\tilde{\Lambda}Xdt\text{ .}  \label{x-e11}
\end{equation}

Using (\ref{bs-02}) we obtain a BSDE that $X=K^{-1}$ must satisfy, that is%
\begin{eqnarray}
dX &=&-X\tilde{\Lambda}XdW+X\left[ K\tilde{A}+\tilde{A}^{\prime }K+\tilde{Q}%
\right] Xdt  \notag \\
&&-X\left[ KB-\tilde{R}+\tilde{\Lambda}\right] X\left[ B^{\prime }K-\tilde{R}%
^{\prime }+\tilde{\Lambda}\right] Xdt+X\tilde{\Lambda}X\tilde{\Lambda}Xdt%
\text{.}  \label{x-e2}
\end{eqnarray}

Setting $Z=-X\tilde{\Lambda}X$, and using the fact that $KX=XK=I$, we obtain
its equation 
\begin{equation}  \label{x-e3}
\begin{array}{rl}
dX = & ZdW+\left[ \tilde{A}X+X\tilde{A}^{\prime }-BXB^{\prime
}+BZ+ZB^{\prime }\right] dt-\left[ Z\tilde{R}^{\prime }X+X\tilde{R}Z\right]
dt \\ 
& \;\;\;+\left[ X\tilde{Q}X+BX\tilde{R}^{\prime }X+X\tilde{R}XB^{\prime }-X%
\tilde{R}X\tilde{R}^{\prime }X\right] dt, \\ 
X(T)= & (R(T)+H)^{-1}.%
\end{array}%
\end{equation}

Notice that all the terms involving $\tilde{\Lambda}$ have been canceled out
thanks to the assumption that the driving noise $W$ is one dimensional; so (%
\ref{x-e3}) no longer contains $\tilde{\Lambda}$. This reveals another
significant feature of the SRE, that is, equation (\ref{x-e3}) for the
inverse matrix $X=K^{-1}$ is itself closed, in the sense that it does not
depend on $K$ or $\tilde{\Lambda}$. Therefore we can solve (\ref{x-e3})
independently without the prior knowledge that $X$ is the inverse of $K$.
Let us call (\ref{x-e3}) \emph{the inverse equation} associated with the SRE
(\ref{bs-01})--(\ref{co-01}).

Therefore, if we are able to solve (\ref{x-e3}) to get $(X,Z)$ with $X>0$ on 
$[0,T]$, then $(P,\Lambda)$, where $P=X^{-1}-R$ and $\Lambda =-X^{-1}ZX-G$,
is a solution to (\ref{bs-01}). In particular, $R+P\equiv K\equiv X^{-1}>0$
is satisfied automatically.

Now we return to the BSDE (\ref{bs-02}) for $K$ and we wish to rewrite it in
terms of $(Z,X)$. There are several ways to do this because of the relations 
$XK=KX=I$, and we will choose one which will serve our propose in this
paper. In (\ref{bs-02}) replace $K^{-1}$ by $X$ and replace $X\tilde{\Lambda}
$ by $-ZK$ to obtain

\begin{equation}  \label{k-n1}
\begin{array}{rl}
dK = & \tilde{\Lambda}dW-\left[ K\tilde{A}+\tilde{A}^{\prime }K+\tilde{Q}%
\right] dt-\left[ KB+\tilde{\Lambda}-\tilde{R}\right] X\tilde{R}^{\prime }dt
\\ 
& \;\;\;+\left[ KB+\tilde{\Lambda}-\tilde{R}\right] \left( XB^{\prime
}-Z\right) Kdt, \\ 
K(T)= & R(T)+H.%
\end{array}%
\end{equation}

We consider (\ref{x-e3}) and (\ref{k-n1}) together as a single system of
BSDEs, and ignore the fact that $X$ is the inverse matrix of $K$ as well as
the constraint $K>0$. This system can be solved one by one: we can solve (%
\ref{x-e3}) first to obtain $(X,Z)$, and then solve (\ref{k-n1}) regarding $%
(X,Z)$ as known parameters. This is actually the approach we will follow.

Under our assumption that $\tilde{R}=0$, our basic BSDEs (\ref{x-e3}) and (%
\ref{k-n1}) are significantly simplified. In fact 
\begin{equation}
dX=ZdW+\left[ \tilde{A}X+X\tilde{A}^{\prime }-BXB^{\prime }+BZ+ZB^{\prime }%
\right] dt+X\tilde{Q}Xdt,\;X(T)=(R(T)+H)^{-1},  \label{d-1-c0}
\end{equation}%
and%
\begin{equation}
dK=\tilde{\Lambda}dW-\left[ K\tilde{A}+\tilde{A}^{\prime }K+\tilde{Q}\right]
dt+(KB+\tilde{\Lambda})\left( XB^{\prime }-Z\right) Kdt,\;K(T)=R(T)+H\text{.}
\label{d-1-c0b}
\end{equation}


\begin{lemma}
\label{lem-a1}Suppose $((X,K),(Z,\tilde{\Lambda}))$ is a bounded solution to
(\ref{d-1-c0}) and (\ref{d-1-c0b}). Then $XK=KX=I$.
\end{lemma}

\begin{proof}
Let $Y=KX-I$. Then applying It\^o's formula to (\ref{d-1-c0}) and (\ref%
{d-1-c0b}) we obtain%
\begin{eqnarray*}
dY &=&K(dX)+(dK)X+\langle dK,dX\rangle \\
&=&UdW+UB^{\prime }dt+Y\tilde{A}^{\prime }dt-\tilde{A}^{\prime }Ydt+Y\tilde{Q%
}Xdt \\
&&+(\tilde{\Lambda}XB^{\prime }-KBZ-\tilde{\Lambda}Z)Ydt \\
&=&UdW+UB^{\prime }dt+Y\left( \tilde{A}^{\prime }+\tilde{Q}X\right) dt \\
&&+\left[ \tilde{\Lambda}XB^{\prime }-\tilde{A}^{\prime }-(KB+\tilde{\Lambda}%
)Z\right] Ydt,
\end{eqnarray*}%
where $U=KZ+\tilde{\Lambda}X$. This is a linear BSDE with the terminal value 
$Y(T)=K(T)X(T)-I=0$. The uniqueness of solution to the linear BSDE then
yields $Y=0$.
\end{proof}

BSDE (\ref{d-1-c0}) is matrix valued with a quadratic term in the drift. If $%
\tilde{Q}>0$, then it is a special case of a \textit{definite} SRE whose
solvability has been established by Bismut \cite{MR0406663} and Peng \cite%
{MR1149069}. In our case where $\tilde{Q}\geq 0$, we use an approximation
scheme, adapted from \cite{MR1149069}, to prove the existence of (\ref%
{d-1-c0}).

\begin{lemma}
\label{lem-p2} Let $\eta $ be a bounded, $n\times n$ symmetric matrix-valued 
$\mathcal{F}_{T}$-measurable random variable. Then there is a unique adapted
bounded solution $(X,Z)$ to the BSDE%
\begin{eqnarray}
dX &=&ZdW+\left( \tilde{A}X+X\tilde{A}^{\prime }-BXB^{\prime }+BZ+ZB^{\prime
}\right) dt+X\tilde{Q}Xdt\text{,}  \notag \\
X(T) &=&\eta \text{. }  \label{x-eq-r1}
\end{eqnarray}%
If in addition $\eta >0$, then $X(t)>0$ for all $t\in \lbrack 0,T]$.
\end{lemma}

\begin{proof}
To show the existence of the BSDE (\ref{x-eq-r1}), we consider the following
iteration%
\begin{eqnarray*}
dX^{(n+1)} &=&Z^{(n+1)}dW+\left[ \tilde{A}X^{(n+1)}+X^{(n+1)}\tilde{A}%
^{\prime }-BX^{(n+1)}B^{\prime }+BZ^{(n+1)}+Z^{(n+1)}B^{\prime }\right] dt \\
&&+X^{(n+1)}\tilde{Q}X^{(n)}dt+X^{(n)}\tilde{Q}X^{(n+1)}dt-X^{(n)}\tilde{Q}%
X^{(n)}dt\text{,} \\
X^{(n+1)}(T) &=&\eta
\end{eqnarray*}%
which is a linear BSDE, whose unique solution defines $(X^{(n+1)},Z^{(n+1)})$%
. Since $\eta >0$, each $X^{(n)}\geq 0$ (Lemma \ref{lem-p1}).

Let $Y^{(n)}=X^{(n)}-X^{(n+1)}$ and $U^{(n)}=Z^{(n)}-Z^{(n+1)}$. Then the
pair $(Y^{(n)},U^{(n)})$ satisfies the following stochastic equation:%
\begin{eqnarray*}
dY^{(n)} &=&U^{(n)}dW+\left[ \tilde{A}Y^{(n)}+Y^{(n)}\tilde{A}^{\prime
}-BY^{(n)}B^{\prime }+BU^{(n)}+U^{(n)}B^{\prime }\right] dt \\
&&+Y^{(n)}\tilde{Q}X^{(n)}dt+X^{(n)}\tilde{Q}Y^{(n)}dt \\
&&+\left[ X^{(n)}\tilde{Q}X^{(n-1)}+X^{(n-1)}\tilde{Q}X^{(n)}-X^{(n)}\tilde{Q%
}X^{(n)}-X^{(n-1)}\tilde{Q}X^{(n-1)}\right] dt \\
Y^{(n)}(T) &=&0\text{.}
\end{eqnarray*}%
Note that, since $\tilde{Q}\geq 0$, the symmetric matrix {\small 
\begin{equation*}
X^{(n)}\tilde{Q}X^{(n-1)}+X^{(n-1)}\tilde{Q}X^{(n)}-X^{(n)}\tilde{Q}%
X^{(n)}-X^{(n-1)}\tilde{Q}X^{(n-1)}\equiv -(X^{(n)}-X^{(n-1)})\tilde{Q}%
(X^{(n)}-X^{(n-1)})\leq 0
\end{equation*}%
} for each $n$, hence $Y^{(n)}\geq 0$ (Lemma \ref{lem-p1}). It follows that $%
X^{(n)}\geq 0$ and $X^{(n)}$ is decreasing in matrix sense, and therefore
has a unique limit denoted by $X$. It is then routine to show that $%
\{Z^{(n)}:n\geq 1\}$ converges as well (in $L^{2}([0,T]\times \Omega
,dt\otimes d\mathbb{P})$) to a limit process $Z$. Then $(X,Z)$ solves (\ref%
{x-eq-r1}).

Since $\tilde{Q}\geq 0$ and $\eta \geq 0$ in matrix sense, $X(t)\geq 0$ for
all $t\in \lbrack 0,T]$. Furthermore, if $\tilde{Q}\geq 0$ and $\eta >0$,
then $X(t)>0$ for any $0\leq t\leq T$. To see this, we apply the
representation (\ref{pos-e1}) to $k=1$, $m=0$, $\hat{A}=\tilde{A}+\frac{1}{2}%
X\tilde{Q}$ and $B=\hat{C}_{1}$. Fix $t\leq T$ and define $\xi _{s}$ by
solving the corresponding SDE\ (\ref{y-li1}). If $p \neq 0$, then, by the
uniqueness of linear BSDE for (\ref{y-li1}) with terminal $\xi _{T}$, we can
conclude that $\xi _{T}\neq 0$. As a result, $\xi _{T}^{\prime }H\xi _{T}>0$%
, and therefore $p ^{\prime }X(t)p>0$ a.s..
\end{proof}


\begin{lemma}
\label{lem-p3}Under the same assumption as in Lemma \ref{lem-p2}, and if in
addition $\eta \geq \delta I$ for some $\delta >0$, then $X^{-1}$ is bounded.
\end{lemma}

\begin{proof}
Since $X$ is bounded, we conclude that 
\begin{equation*}
\beta _{0}=\text{ess}\sup_{\omega \in \Omega ,s\leq T}\{-\inf_{|\xi
|=1}\langle 2\tilde{A}\xi +X\tilde{Q}\xi ,\xi \rangle ,0\}
\end{equation*}%
is finite. It then follows from Lemma \ref{lem-p1a} that $X(t)\geq \delta
e^{-\beta _{0}T}I$ so that $X^{-1}$ is bounded.
\end{proof}

\begin{lemma}
\label{main-lemma} Let $\eta =\left( R(T)+H\right) ^{-1}\geq \delta I$ in
Lemma \ref{lem-p2} and $(X,Z)$ be the corresponding solution. Then $%
P=X^{-1}-R$ solves the SRE (\ref{bs-01})--(\ref{co-01}).
\end{lemma}

\begin{proof}
Let $K=X^{-1}$ which have been proved to be a bounded matrix valued
semimartingale. Let $\tilde{\Lambda}=-KZK$. Since $KX=XK=I$, it follows from
Lemma \ref{lem-a1} that $(K,\tilde{\Lambda})$ satisfies (\ref{k-n1}).
Therefore $(P,\Lambda )$, where $P=X^{-1}-R$ and $\Lambda =\tilde{\Lambda}-G$%
, in turn solves (\ref{bs-01}). Moreover, $K=X^{-1}>0$, namely, the
constraint (\ref{co-01}) is satisfied.
\end{proof}

The uniqueness for SRE is well known. The proof of Theorem \ref{main-th} is
complete.

\section{Discussions and Examples}

\label{example}

In this section we discuss about the assumptions of Theorem \ref{main-th}
and give examples for illustration.


First of all, if $R\equiv 0$, then condition (\ref{c-a}) holds
automatically, and (\ref{c-b}) is equivalent to $Q\geq 0$. In addition,
assumption (ii) in Theorem \ref{main-th} boils down to $H>0$. (The condition 
$H^{-1}\geq \delta I$ is implied by $H>0$ and the fact that $H$ is bounded.)
In this case, our result improves Theorem 5.2 in \cite{MR1982738}, since
here we do not need to assume $C=0$ and $H^{-1}$ is bounded.

If $R$ is a constant yet indefinite matrix, and $B=C=0$, then again (\ref%
{c-a}) is satisfied, and (\ref{c-b}) reduces to $Q-RA-A^{\prime }R\geq0$. In
this case we recover Theorem 5.3 of \cite{MR1982738}. However, from Theorem %
\ref{main-th} we immediately realize that the assumption $B=C=0$ is far from
being essential. Indeed, in the case when $R,B$ and $C$ are non-random
matrices, the essential condition is $RB+C^{\prime }R=0$, which can be
satisfied easily by infinitely many non-zero matrices $B$ and $C$ and
indefinite matrices $R$. In this case, the condition $Q-RA-A^{\prime }R\geq0$
should be replaced by $Q+\dot R+C^{\prime }RC+R\left( BC-A\right) +\left(
C^{\prime }B^{\prime }-A^{\prime }\right) R\geq 0$, where $\dot R$ denotes
the derivative of $R$, which is zero if $R$ is a constant matrix.

As a matter of fact, we can ``generate" many generally indefinite, adapted
processes $R$ satisfying condition (\ref{c-a}). To see this, let $S$ be the
solution of the following matrix-valued, sample-wise ODE 
\begin{equation*}
dS=\left[e^{C^{\prime }W}Fe^{BW}-C^{\prime }SB-(C^{\prime 2})S-SB^2\right]dt
\end{equation*}
with any given initial state, where $F$ is any given adapted process so that 
$e^{C^{\prime }W}Fe^{BW}$ is integrable over $t\in[0,T]$ a.s.. Define $%
R=e^{-C^{\prime }W}Se^{-BW}$. Then It\^o's formula yields that $R$ satisfies
(\ref{keycondition}), namely, (\ref{c-a}) holds.

If condition (\ref{c-a}) does not hold (i.e. $\tilde R\neq 0$), then, we
will need to study the general inverse equation, (\ref{x-e3}). This is a
very interesting BSDE, since it is matrix-valued involving a cubic term of $%
X $. In general, its global existence is not guaranteed. For example,
suppose that $\tilde{R}=I$, $\tilde{A}=0$, $B=0$, and that the terminal and $%
Q$ are non-random. Then (\ref{x-e3}) becomes 
\begin{equation}
dX=\left[ X\tilde{Q}X-XXX\right] dt  \notag
\end{equation}%
whose solution may explode in a finite time (and thus the corresponding SRE
can not have a global solution). An interesting and challenging open problem
is to identify the ``weakest" condition on $\tilde R$ so that the cubic BSDE
(\ref{x-e3}) admits a solution.

On the other hand, the condition $\left( R(T)+H\right) ^{-1}\geq \delta I$
for some $\delta >0$ is more technical than essential. One can weaken this
condition by incorporating more involved technicalities in our analysis.
However, the main goal of this note is to introduce and highlight the main
approach, that is to use a system of BSDEs to substitute the original mix of
a BSDE and an algebraic constraint, to solving the indefinite SRE.
Therefore, we have preferred not to let undue technicalities distract the
main idea.

Finally, let us remark that the proof of Lemma \ref{lem-p2}, along with
Lemma \ref{main-lemma}, has indeed suggested a numerical scheme to solve the
indefinite SRE. %
%


\end{document}